\documentclass[11pt]{article}

\usepackage[hidelinks]{hyperref}
\usepackage{amsmath}
\usepackage{amssymb}
\usepackage{amsthm}
\usepackage{mathtools}
\usepackage{fullpage}

\newtheorem{theorem}{Theorem}[section]
\newtheorem{lemma}[theorem]{Lemma}
\newtheorem{corollary}[theorem]{Corollary}
\newtheorem{claim}[theorem]{Claim}
\newtheorem{definition}[theorem]{Definition}

\DeclareMathOperator{\spn}{span}
\DeclareMathOperator{\tr}{tr}
\DeclareMathOperator{\diag}{diag}
\DeclareMathOperator{\rank}{rank}
\DeclareMathOperator{\Rea}{Re}
\DeclareMathOperator{\Ima}{Im}
\DeclareMathOperator{\Proj}{Proj}
\renewcommand{\vec}[1]{\boldsymbol{#1}}

\newcommand{\ip}[2]{\langle #1,#2 \rangle}

\newcommand{\R}{\mathbb{R}}
\newcommand{\C}{\mathbb{C}}

\newcommand{\Z}{\mathbb{Z}}
\newcommand{\F}{\mathbb{F}}

\title{\bf  Sylvester-Gallai for Arrangements of Subspaces}
\author{Zeev Dvir\thanks{Department of Computer Science and Department of Mathematics, Princeton University, Princeton NJ.
Email: \texttt{zeev.dvir@gmail.com}.}
\and 
Guangda Hu\thanks{Department of Computer Science  Princeton University, Princeton NJ.
Email: \texttt{guangdah@cs.princeton.edu}.}
}
\date{}

\begin{document}
\maketitle

\begin{abstract}
In this work we study arrangements of $k$-dimensional subspaces $V_1,\ldots,V_n \subset \C^\ell$. Our main result shows that, if every pair $V_{a},V_b$ of subspaces is contained in a dependent triple (a triple $V_{a},V_b,V_c$ contained in a  $2k$-dimensional space), then the entire arrangement must be contained in a subspace whose dimension depends only on $k$ (and not on $n$). The theorem holds under the assumption that $V_a \cap V_b = \{0\}$ for every pair (otherwise it is false). This generalizes the Sylvester-Gallai  theorem (or Kelly's theorem for complex numbers), which proves the $k=1$ case. Our proof also handles arrangements in which we  have many pairs (instead of all) appearing in dependent triples, generalizing the quantitative results of Barak et. al. \cite{BDWY-pnas}.

One of the main ingredients in the proof is a strengthening of a Theorem of Barthe \cite{Bar98} (from the $k=1$ to $k>1$ case)  proving the existence of a linear map that makes the  angles between pairs of subspaces large on average. Such a mapping can be found, unless there is an obstruction in the form of a low dimensional subspace intersecting many of the spaces in the arrangement (in which case one can use a different argument to prove the main theorem).
\end{abstract}

\section{Introduction}

The Sylvester-Gallai (SG) theorem states that for $n$  points $\vec{v}_1,\vec{v}_2,\ldots,\vec{v}_n \in \R^\ell$, if for every pair $\vec{v}_i,\vec{v}_j$ there is a third point $\vec{v}_k$ on the line passing through $\vec{v}_i,\vec{v}_j$, then all points must lie on a single line. This was first posed by Sylvester~\cite{Syl93}, and was solved by Melchior~\cite{Mel40}. It was also conjectured independently by Erd{\"o}s \cite{Erd43} and proved shortly after by Gallai. We refer the reader to the survey~\cite{BM90} for more information about the history and various generalizations of this theorem. The complex version of this theorem was proved by Kelly~\cite{Kel86} (see also \cite{EPS06, DSW12} for alternative proofs) and states that if $\vec{v}_1,\vec{v}_2,\ldots,\vec{v}_n \in \C^\ell$  and for every pair $\vec{v}_i,\vec{v}_j$ there is a third $\vec{v}_k$ on the same complex line, then all points are contained in some complex plane (over the complex numbers, there are planar examples and so this theorem is tight).

In~\cite{DSW12} (based on earlier work in \cite{BDWY-pnas}), the following quantitative variant of the SG theorem was proved. For a set $S \subset \C^\ell$ we denote by $\dim(S)$  the smallest $d$ such that $S$ is contained in a $d$-dimensional subspace of $\C^\ell$.

\begin{theorem}[\cite{DSW12}] \label{thm:osg}
Given $n$ points $\vec{v}_1,\vec{v}_2,\ldots,\vec{v}_n \in \C^\ell$, if for every $i\in[n]$ there exists at least $\delta n$ values of $j\in[n]\setminus\{i\}$ such that the line through $\vec{v}_i$ and $\vec{v}_j$ contains a third point $\vec{v}_k$, then $\dim\{\vec{v}_1,\vec{v}_2,\ldots,\vec{v}_n\}\leq 10/\delta$.
\end{theorem}

(The dependence on $\delta$ is asymptotically tight). From here on, we will work with homogeneous subspaces (passing through zero) instead of affine subspaces (lines/planes etc). The difference is not crucial to our results and the affine version can always be derived by intersecting with a generic hyperplane. In this setting, the above theorem will be stated for a set of one-dimensional subspaces, each spanned by some  $\vec{v}_i$ (and no two $\vec{v}_i$'s being a multiple of each other) and collinearity of $\vec{v}_i,\vec{v}_j,\vec{v}_k$  is replaced with the three vectors being linearly dependent (i.e., contained in a 2-dimensional subspace).

One natural high dimensional variant of the SG theorem, studied in \cite{Han65, BDWY-pnas}, replaces 3-wise dependencies with $t$-wise dependencies (e.g, every triple is in some coplanar four-tuple). In this work, we raise another natural  high-dimensional variant in which the {\em points} themselves are replaced with $k$-dimensional subspaces. We consider such arrangements with many 3-wise dependencies (defined appropriately) and attempt to prove that the entire arrangement lies in some low dimensional space. We will consider arrangements $V_1,\ldots,V_n \subset \C^\ell$ in which each $V_i$ is $k$-dimensional and with each pair satisfying  $V_{i_1} \cap V_{i_2} = \{\vec{0}\}$. A dependency can then be defined as a triple $V_{i_1},V_{i_2},V_{i_3}$ of $k$-dimensional subspaces that are contained in a single $2k$-dimensional subspace. The pair-wise zero intersections guarantee that every pair of subspaces defines a unique $2k$-dimensional space (their span) and so, this definition of dependency behaves in a similar way to collinearity. For example, we have that if $V_{i_1},V_{i_2},V_{i_3}$ are dependent and $V_{i_2},V_{i_3},V_{i_4}$ are dependent then also $V_{i_1},V_{i_2},V_{i_4}$ are dependent. This would not hold if  we allowed some pairs to have non zero intersections. In fact, if we allow non-zero intersection then we can construct an arrangement of two dimensional spaces with many dependent triples and with dimension as large as $\sqrt{n}$ (see below). We now state our main theorem, generalizing Theorem~\ref{thm:osg} (with slightly worse parameters) to the case $k>1$. We use the standard $V + U$ notation to denote the subspace spanned by all vectors in $V \cup U$. We use big `O' notation to hide absolute constants.

\begin{theorem} \label{thm:sg}
Let $V_1,V_2,\ldots,V_n\subset \C^\ell$ be $k$-dimensional subspaces such that $V_{i}\cap V_{i'}=\{\vec{0}\}$ for all $i\neq i'\in[n]$. Suppose that, for every $i_1\in[n]$ there exists at least $\delta n$ values of $i_2\in[n]\setminus\{i_1\}$ such that $V_{i_1}+V_{i_2}$ contains some $V_{i_3}$ with $i_3 \not\in \{i_1,i_2\}$. Then $$\dim(V_1+V_2+\cdots+V_n)= O(k^4/\delta^2).$$
\end{theorem}

The condition $V_i \cap V_{i'} = \{\vec{0}\}$ is needed due to the following example. Set $k=2$ and $n=\ell(\ell-1)/2$ and let $\{\vec{e}_1,\vec{e}_2,\ldots,\vec{e}_\ell\}$ be the standard basis of $\R^\ell$. Define the $n$ spaces to be  $V_{ij} = \spn\{\vec{e}_i,\vec{e}_j\}$  with $1\leq i<j\leq\ell$. Now, for each $(i,j) \neq (i',j')$ the sum $V_{ij} + V_{i'j'}$ will contain a third space (since the size of $\{i,j,i',j'\}$ is at least three). However, this arrangement has dimension $\ell > \sqrt{n}$.

The bound $O(k^4/\delta^2)$ is probably not tight and we conjecture that it could be improved to $O(k/\delta)$, possibly with a modification of our proof. One can always construct an arrangement with dimension $2k/\delta$  by partitioning the subspaces into $1/\delta$ groups, each contained in a single $2k$ dimensional space. 

\subparagraph{Overview of the proof:}
A preliminary observation is that it suffices to prove the theorem over $\R$. This is because an arrangement of $k$-dimensional Complex subspaces can be translated into an arrangement of $2k$-dimensional Real subspaces (this is proved at the end of Section~\ref{sec:adsystem}). Hence, we will now focus on Real arrangements.

The proof of the theorem is considerably simpler when the arrangement of subspaces $V_1,\ldots,V_n$ satisfies an extra `robustness' condition, namely that every two spaces have an angle bounded away from zero. More formally, if for every two unit vectors $\vec{v}_1 \in V_{i_1}$ and $\vec{v}_2 \in V_{i_2}$ we have $|\ip{\vec{v}_1}{\vec{v}_2}| \leq1-\tau$ for some absolute constant $\tau>0$. This condition implies that, when we have a dependency of the form $V_{i_3} \subset V_{i_1} + V_{i_2}$, every unit vector in $V_{i_3}$ can be obtained as a linear combination {\em with bounded coefficients} (in absolute value) of unit vectors from $V_{i_1},V_{i_2}$. Fixing an orthogonal basis for each subspace and using the conditions of the theorem, we are able to construct many local linear dependencies between the basis elements. We then show (using the bound on the coefficients in the linear combinations) that the space of linear dependencies between all basis vectors, considered as a subspace of $\R^{kn}$, contains the rows of an $nk \times nk$ matrix that has large entries on the diagonal and small entries off the diagonal. Since matrices of this form have high rank (by a simple spectral argument), we conclude that the original set of basis vectors must have small dimension.

To handle the general case, we show that, unless some low dimensional subspace $W$ intersects many of the spaces $V_i$ in the arrangement, we can find a change of basis that makes the angles between the spaces large on average (in which case, the previous argument works). This gives us the overall strategy of the proof: If  such a $W$ exists, we project $W$ to zero and continue by induction. The loss in the overall dimension is bounded by the dimension of $W$, which can be chosen to be small enough. Otherwise (if such $W$ does not exist) we apply the change of basis and use it to bound the dimension.

The change of basis is found by generalizing a theorem  of Barthe~\cite{Bar98} (see \cite{DSW14} for a more accessible treatment) from the one dimensional case  (arrangement of points) to higher dimension. We state this result here since we believe it could be of independent interest. To state the theorem we must first introduce the following, somewhat technical, definition.

\begin{definition}[admissible basis set, admissible basis vector] \label{def:adbasic}
Given a list of vector spaces $\mathcal{V}=(V_1,V_2,\ldots,V_n)$ ($V_i\subseteq\R^\ell$), a set $H\subseteq [n]$ is called a {\em $\mathcal{V}$-admissible basis set} if
 $$\dim(\sum_{i\in H}V_i)=\sum_{i\in H}\dim(V_i)=\dim(\sum_{i\in[n]}V_i),$$ i.e. if every space with index in $H$ has intersection $\{\vec{0}\}$ with the span of the other spaces with indices in $H$, and the spaces with indices in $H$ span the entire space $\sum_{i\in[n]}V_i$.
  
A {\em $\mathcal{V}$-admissible basis vector} is any indicator vector $\vec{1}_H$ of some $\mathcal{V}$-admissible basis set $H$ (where the $i$-th entry of $\vec{1}_H$ equals $1$ if $i\in H$ and $0$ otherwise). 
\end{definition}

The following theorem is proved in Section~\ref{sec:barthe}.
\begin{theorem} \label{thm:barthe}
Given a list of vector spaces $\mathcal{V}=(V_1,V_2,\ldots,V_n)$ ($V_i\subseteq\R^\ell$) with $V_1+V_2+\cdots+V_n=\R^\ell$ and a vector $\vec{p}\in\R^n$ in the convex hull of all $\mathcal{V}$-admissible basis vectors. Then for any $\varepsilon>0$, there exists an invertible linear map $M:\R^\ell\mapsto\R^\ell$ such that
$$\Big\|\sum_{i=1}^np_i\Proj_{M(V_i)}-I_{\ell\times\ell}\Big\|\leq\varepsilon,$$
where $\|\cdot\|$ is the spectral norm and $\Proj_{M(V_i)}$ is the orthogonal projection matrix onto $M(V_i)$.
\end{theorem}

The connection to the explanation given in the proof overview is as follows: If there is no subspace $W$ of low dimension that intersects many of the spaces $V_1,\ldots,V_n$ then, one can show that there exists a vector $\vec{p}$ in the convex hull of all $\mathcal{V}$-admissible basis vectors such that the entries of $\vec{p}$ are not too small. This is enough to show that the average angle between pairs of spaces is large since otherwise one can derive a contradiction to the inequality which says that the sum of orthogonal projections of any unit vector must be relatively small.

 The proof of the one dimensional case in \cite{Bar98} (which does not have an $\epsilon$ error) proceeds by defining a strictly convex function $f(t_1,\ldots,t_m)$ on $\R^m$ and shows that the function is bounded. This means that there must exist a point in which all partial derivatives of $f$ vanish. Solving the resulting equations gives an invertible matrix that defines the required change of basis. We follow a similar strategy, defining an appropriate bounded function $f(t_1,\ldots,t_m,R_1,\ldots,R_n)$ in more variables, where the extra variables $R_1,\ldots,R_n$ represent the action of the orthogonal group $\mathbf{O}(k)$ on each of the spaces. However, in our case, we cannot show that $f$ is strictly convex and so a maximum might not exist. However, we are still able to show that there exists a point in which all partial derivatives are very small (smaller than any $\epsilon>0$), which is sufficient for our purposes.

\subparagraph{Connection to Locally Correctable Codes.}

A $q$-query Locally Correctable Code (LCC) over a field $\F$ is a $d$-dimensional subspace  $C \subset \F^n$ that allows for `local correction' of codewords (elements of $C$) in the following sense. Let $\vec{y} \in C$ and suppose we have query access to $\vec{y}'$ such that $\vec{y}_i = \vec{y}'_i$ for at least $(1-\delta)n$ indices $i \in [n]$ (think of $\vec{y}'$ as a noisy version of $\vec{y}$). Then, for every $i$, we can probabilistically pick $q$ positions in $\vec{y}'$ and, from their (possibly incorrect values), recover the correct value of $\vec{y}_i$ with high probability (over the choice of queries). LCC's play an important role in  theoretical computer science (mostly over finite fields but recently also over the Reals, see \cite{Dvir-rigidity}) and are still poorly understood. In particular, when $q$ is constant greater than 2, there are exponential gaps between the dimension  of explicit constructions and the proven upper bounds. In \cite{BDWY11} it was observed that $q$-LCCs are essentially equivalent to configurations of points with many local dependencies\footnote{One important difference is that LCC's give rise to configurations where each point can repeat more than once.}. A variant of Theorem~\ref{thm:osg} shows for example that the maximal dimension of a $2$-LCC in $\R^n$ has dimension bounded by $(1/\delta)^{O(1)}$. Our results can be interpreted in this framework as dimension upper bounds for $2$-query LCC's in which each coordinate is replaced by a `block' of $k$ coordinate. Our results then show that, even under this relaxation, the dimension still cannot increase with $n$. The case of $3$-query LCC's over the Reals is still wide open (some modest progress was made recently in \cite{DSW14}) and we hope that the methods developed in this work could lead to further progress on this tough problem.

\subparagraph{Organization.} In Section~\ref{sec:adsystem}, we  define the notion of $(\alpha,\delta)$-systems (which generalizes the SG condition) and reduce our $k$-dimensional Sylvester-Gallai theorem to a more general theorem, Theorem~\ref{thm:main}, on the dimension of $(\alpha,\delta)$-systems (this part also includes the reduction from Complex to Real arrangements). Then, in Section~\ref{sec:barthe}, we prove the generalization of Barthe's theorem (Theorem~\ref{thm:barthe}). Finally,  in section~\ref{sec:main}, we prove the our main result regarding $(\alpha,\delta)$-systems.

\section{Reduction to \texorpdfstring{$(\alpha,\delta)$}{(alpha,delta)}-systems} \label{sec:adsystem}

The notion of an $(\alpha,\delta)$-system is used to `organize' the dependent triples in the arrangement in a more convenient form so that each space is in many triples and every pair of spaces is together only in a few dependent triples. We also allow dependent {\em pairs} as those might arise  when we apply a linear map on the arrangement.

\begin{definition}[$(\alpha,\delta)$-system] \label{def:adsystem}
Given a list of vector spaces $\mathcal{V}=(V_1,V_2,\ldots,V_n)$ ($V_i\subseteq\R^\ell$), we call a list of sets $\mathcal{S}=(S_1,S_2,\ldots,S_w)$  an {\em $(\alpha,\delta)$-system} of $\mathcal{V}$ ($\alpha\in\Z^+$, $\delta>0$) if
\begin{enumerate}
\item Every $S_j$ is a subset of $[n]$ of size either $3$ or $2$.
\item If $S_j$ contains $3$ elements $i_1$, $i_2$ and $i_3$, then $V_{i_1}\subseteq V_{i_2}+V_{i_3}$, $V_{i_2}\subseteq V_{i_1}+V_{i_3}$ and $V_{i_3}\subseteq V_{i_1}+V_{i_2}$. If $S_j$ contains $2$ elements $i_1$ and $i_2$, then $V_{i_1}=V_{i_2}$.
\item Every $i\in[n]$ is contained in at least $\delta n$ sets of $\mathcal{S}$.
\item Every pair $\{i_1,i_2\}$ ($i_1\neq i_2\in[n]$) appears together in at most $\alpha$ sets of $\mathcal{S}$.
\end{enumerate}
\end{definition}

Note that we allow $\delta>1$ in an $(\alpha,\delta)$-systems. This is different from the statement of the Sylvester-Gallai theorem where $\delta\in[0,1]$. We have the following simple observations.

\begin{lemma} \label{lem:awd}
Let $\mathcal{S}=(S_1,S_2,\ldots,S_w)$ be an $(\alpha,\delta)$-system of some vector space list $\mathcal{V}$. Then $\delta n^2/3\leq w\leq\alpha n^2/2$ and $\delta/\alpha\leq3/2$.
\end{lemma}

\begin{proof}
We consider the sum $\sum_{j\in[w]}|S_j|$. By the definition of $(\alpha,\delta)$-system,
$$n\cdot\delta n\leq\sum_{j\in[w]}|S_j|\leq3w\quad\Longrightarrow\quad\delta n^2/3\leq w.$$
Then we consider the number of pairs $\sum_{j\in[w]}\binom{|S_j|}{2}$, we can see
$$w\leq\sum_{j\in[w]}\binom{|S_j|}{2}\leq\alpha\binom{n}{2}\leq\alpha n^2/2.$$
It follows that $\delta/\alpha\leq3/2$.
\end{proof}

\begin{lemma} \label{lem:removebad}
Let $\mathcal{V}=(V_1,V_2,\ldots,V_n)$ ($V_i\subseteq\R^\ell$) be a list of vector spaces and $\mathcal{S}=(S_1,S_2,\ldots,S_w)$ be a list of sets. If $w\geq\delta n^2$ and $\mathcal{S}$ satisfies the first, second and fourth requirements in Definition~\ref{def:adsystem}, then there exists a sublist $\mathcal{V}'$ of $\mathcal{V}$ and a sublist $\mathcal{S}'$ of $\mathcal{S}$ such that $|\mathcal{V}'|\geq\delta n/(2\alpha)$ and $\mathcal{S}'$ is an $(\alpha,\delta/2)$-system of $\mathcal{V}'$.
\end{lemma}

\begin{proof}
 We iteratively remove all $V_i$'s that appear in less than $\delta n/2$ sets, and the sets they appear in. There are $n$ $V_i$'s in total, so we can remove at most $n\cdot\delta n/2$ sets. When the procedure ends, we still have at least $\delta n^2-\delta n^2/2\geq\delta n^2/2$ sets. So we do not remove all of $V_1,V_2,\ldots,V_n$. For a remaining $V_i$, since it appears in at least $\delta n/2$ sets, we must still have at least $\delta n/(2\alpha)$ vector spaces left. Let $\mathcal{V}'$ be the list of these spaces and $\mathcal{S}'$ be the list of the remaining sets. We can see that $\mathcal{S}'$ is an $(\alpha,\delta/2)$-system of $\mathcal{V}'$.
\end{proof}

\begin{lemma} \label{lem:linearmap}
Let $\mathcal{V}=(V_1,V_2,\ldots,V_n)$ ($V_i\subseteq\R^\ell$) be a list of vector spaces with an $(\alpha,\delta)$-system $\mathcal{S}=(S_1,S_2,\ldots,S_w)$. Then for any linear map $P:\R^\ell\mapsto\R^\ell$, $\mathcal{S}$ is also an $(\alpha,\delta)$-system of $\mathcal{V}'=(V_1',V_2',\ldots,V_n')$, where $V_i'=P(V_i)$.
\end{lemma}

\begin{proof}
This is trivial since, if $V_{i_1}\subseteq V_{i_2}+V_{i_3}$, then
\begin{equation*}
V_{i_1}'=P(V_{i_1})\subseteq P(V_{i_2}+V_{i_3})=P(V_{i_2})+P(V_{i_3})=V_{i_2}'+V_{i_3}'.\qedhere
\end{equation*}
\end{proof}

\begin{lemma} \label{lem:removezero}
Let $\mathcal{V}=(V_1,V_2,\ldots,V_n)$ ($V_i\subseteq\R^\ell$) be a list of vector spaces with an $(\alpha,\delta)$-system $\mathcal{S}=(S_1,S_2,\ldots,S_w)$. Suppose we remove all zero ($\{\vec{0}\}$) spaces in $\mathcal{V}$ in the following way:
\begin{enumerate}
\item Let $n'$ be the number of nonzero (not $\{\vec{0}\}$) vector spaces in $\mathcal{V}$, and $\phi$ be a one-to-one mapping from the indices of nonzero spaces to $[n']$. We define $V_1'=V_{\phi^{-1}(1)}$, $V_2'=V_{\phi^{-1}(2)}$,\ldots, $V_{n'}'=V_{\phi^{-1}(n')}$ to be all the nonzero spaces.
\item For each $S_j$ ($j\in[w]$), we define $S_j'=\{\phi(i):i\in S_j,V_i\neq\{\vec{0}\}\}$.
\item We remove the $S_j$'s that are empty.
\end{enumerate}
Let $\mathcal{S}'$ be the list of the remaining sets in $S_1',S_2',\ldots,S_w'$. Then $\mathcal{S}'$ is an $(\alpha,\delta')$-system of $\mathcal{V}'=(V_1',V_2',\ldots,V_{n'}')$, where $\delta'=\delta n/n'$.
\end{lemma}

\begin{proof}
We first consider an $S_j$ containing 3 elements $i_1$, $i_2$ and $i_3$. If none of $V_{i_1}$, $V_{i_2}$, $V_{i_2}$ is $\{\vec{0}\}$, we have $S_j'=\{\phi(i_1),\phi(i_2),\phi(i_3)\}$ and it satisfies the second requirement. If exactly one of them is $\{\vec{0}\}$, say $V_{i_3}=\{\vec{0}\}$, we can see $S_j'=\{\phi(i_1),\phi(i_2)\}$ and $V_{\phi(i_1)}'=V_{\phi(i_2)}'$ by $V_{i_1}\subseteq V_{i_2}+\{\vec{0}\}$, $V_{i_2}\subseteq V_{i_1}+\{\vec{0}\}$. If exactly two of them are $\{\vec{0}\}$, say $V_{i_2}=V_{i_3}=\{\vec{0}\}$, then $V_{i_1}\subseteq V_{i_2}+V_{i_3}$ must also be $\{\vec{0}\}$, contradiction. If all $V_{i_1}$, $V_{i_2}$, $V_{i_3}$ are $\{\vec{0}\}$, $S_j'=\emptyset$ and it is removed.

We then consider an $S_j$ containing 2 elements $i_1$ and $i_2$. If neither of $V_{i_1}$, $V_{i_2}$ is $\{\vec{0}\}$, we have $S_j'=\{\phi(i_1),\phi(i_2)\}$ and $V_{\phi(i_1)}'=V_{\phi(i_2)}'$. If one of them is $\{\vec{0}\}$, the other must be $\{\vec{0}\}$ by $V_{i_1}=V_{i_2}$, and $S_j'$ is removed.

In summary, the first two requirements of the definition of $(\alpha,\delta)$-system are satisfied. We can also see that each $i\in[n']$ is contained in at least $\delta'n'=\delta n$ sets and each pair $\{i_1,i_2\}$ with $i_1\neq i_2\in[n']$ is contained in at most $\alpha$ sets, because we have only removed the sets containing only indices of zero spaces. Therefore the third and fourth requirements are also satisfied.
\end{proof}

Combining the above two lemmas, we have the following corollary.

\begin{corollary} \label{cor:remove}
Let $\mathcal{V}=(V_1,V_2,\ldots,V_n)$ ($V_i\subseteq\R^\ell$) be a list of vector spaces with an $(\alpha,\delta)$-system, and $P:\R^\ell\mapsto\R^\ell$ be any linear map. Define $\mathcal{V}'=(V_1',V_2',\ldots,V_{n'}')$ to be the list of nonzero spaces in $P(V_1),P(V_2),\ldots,P(V_n)$. Then $\mathcal{V}'$ has an $(\alpha,\delta')$-system, where $\delta'=\delta n/n'$.
\end{corollary}

Theorem~\ref{thm:sg}, will be  derived from the following, more general statement, saying that the dimension $d$ is small if there is a $(\alpha,\delta)$-system.

\begin{definition}[$k$-bounded]
A vector space $V\subseteq\R^\ell$ is {\em $k$-bounded} if $\dim V\leq k$.
\end{definition}

\begin{theorem} \label{thm:main}
Let $\mathcal{V}=(V_1,V_2,\ldots,V_n)$ ($V_i\subseteq\R^\ell$) be a list of $k$-bounded vector spaces with an $(\alpha,\delta)$-system and $d=\dim(V_1+V_2+\cdots+V_n)$, then $d=O(\alpha^2 k^4/\delta^2)$.
\end{theorem}

We can easily reduce the high dimensional Sylvester-Gallai problem in $\C^\ell$ (Theorem~\ref{thm:sg}) to the setting of Theorem~\ref{thm:main} in $\R^\ell$ as shown below.

\begin{proof}[Proof of Theroem~\ref{thm:sg} using Theorem~\ref{thm:main}]
Let $B_j=\{\vec{v}_{j1},\vec{v}_{j2},\ldots,\vec{v}_{jk}\}$ be a basis of $V_j$. Define
$$V_j'=\spn\big\{\Rea(\vec{v}_{j1}),\Rea(\vec{v}_{j2}),\ldots,\Rea(\vec{v}_{jk}),\Ima(\vec{v}_{j1}),\Ima(\vec{v}_{j2}),\ldots,\Ima(\vec{v}_{jk})\big\}\quad\forall j\in[n].$$

\begin{claim} \label{clm:complex}
$V_j'=\{\Rea(\vec{v}):\vec{v}\in V_j\}$ for every $j\in[n]$.
\end{claim}

\begin{proof}
For every $\vec{v}'\in V_j'$, there exist $\lambda_1,\lambda_2,\ldots,\lambda_k,\mu_1,\mu_2,\ldots,\mu_k\in\R$ such that
\begin{equation*}\begin{split}
\vec{v}'&=\sum_{s=1}^k\Big(\lambda_s\Rea(\vec{v}_{js})+\mu_s\Ima(\vec{v}_{js})\Big)=\sum_{s=1}^k\Big(\lambda_s\Rea(\vec{v}_{js})+\mu_s\Rea(-i\vec{v}_{js})\Big) \\
& =\Rea\left(\sum_{s=1}^k(\lambda_s-i\mu_s)\vec{v}_{js}\right).
\end{split}\end{equation*}
Since $\lambda_1,\lambda_2,\ldots,\lambda_k,\mu_1,\mu_2,\ldots,\mu_k$ can take all values in $\R$, we can see the claim is proved.
\end{proof}

\begin{claim}[{\cite[Lemma~2.1]{BDWY-pnas}}] \label{clm:triple}
Given a set $A$ with $r\geq3$ elements, we can construct a family of  $r^2-r$ triples of elements in $A$ with following properties: 1) Every triple contains three distinct element; 2) Every element of $A$ appears in exactly $3(r-1)$ triples; 3) Every pair of two distinct elements in $A$ is contained together in at most $6$ triples.
\end{claim}

We call a $2k$-dimensional subspace $U \subset \C^\ell$  {\em special} if it contains at least three of $V_1,V_2,\ldots,V_n$. We define the {\em size} of a special space as the number of spaces among $V_1,V_2,\ldots,V_n$ contained in it. For a special space with size $r$, we add the $r^2-r$ triples of indices of the spaces in it with the properties in Claim~\ref{clm:triple}. Let $\mathcal{S}$ be the family of all these triples. We claim that $\mathcal{S}$ is a $(6,3\delta)$-system of $\mathcal{V}=(V_1',V_2',\ldots,V_n')$.

For every triple $\{j_1,j_2,j_3\}\in\mathcal{S}$, we can see that $V_{j_1},V_{j_2},V_{j_3}$ are contained in the same $2k$-dimensional special space. And by $V_{j_1}\cap V_{j_2}=\{\vec{0}\}$, the space must be $V_{j_1}+V_{j_2}$ and hence $V_{j_3}\subseteq V_{j_1}+V_{j_2}$. By Claim~\ref{clm:complex},
$$V_{j_3}'=\{\Rea(\vec{v}):\vec{v}\in V_{j_3}\}\subseteq\{\Rea(\vec{u})+\Rea(\vec{w}):\vec{u}\in V_{j_1},\vec{w}\in V_{j_2}\}=V_{j_1}'+V_{j_2}'.$$
Similarly, $V_{j_1}'\subseteq V_{j_2}'+V_{j_3}'$ and $V_{j_2}'\subseteq V_{j_1}'+V_{j_3}'$. One can see every pair in $[n]$ appears in at most $6$ triples because the corresponding two spaces are contained in at most one special space, and the pair appears at most $6$ times in the triples constructed from this special space. For every $j\in[n]$, there are at least $\delta n$ values of $j'\in[n]\setminus\{j\}$ such that there is a special space containing $V_j$ and $V_{j'}$. This implies that the number of triples that $j$ appears in is
$$\sum_{\text{special space }U\atop V_j\subseteq U}3\big(\operatorname{size}(U)-1\big)=3\sum_{\text{special space }U\atop V_j\subseteq U}\big|\{j'\neq j:V_{j'}\subseteq U\}\big|\geq3\delta n.$$
Therefore $\mathcal{S}$ is a $(6,3\delta)$-system of $\mathcal{V}$. By Theorem~\ref{thm:main},
$$\dim(V_1'+V_2'+\cdots+V_n')=O(6^2(2k)^4/(3\delta)^2)=O(k^4/\delta^2).$$
Note that
\begin{equation*}\begin{split}
V_1+V_2+\cdots+V_n & \subseteq \spn\big\{\Rea(\vec{v}_{js}),\Ima(\vec{v}_{js})\big\}_{j\in[n],s\in[k]}\quad(\text{span with complex coefficients}), \\
V_1'+V_2'+\cdots+V_n' & = \spn\big\{\Rea(\vec{v}_{js}),\Ima(\vec{v}_{js})\big\}_{j\in[n],s\in[k]}\quad(\text{span with real coefficients}).
\end{split}\end{equation*}
We thus have $\dim(V_1+V_2+\cdots+V_n)\leq\dim(V_1'+V_2'+\cdots+V_n')=O(k^4/\delta^2)$.
\end{proof}

\section{A generalization of Barthe's Theorem} \label{sec:barthe}

We prove Theorem~\ref{thm:barthe} in the following 3 subsections. In the fourth and last subsection, we state a convenient variant of the theorem  (Theorem~\ref{thm:barthec}) that will be used later in the proof of our main result. The idea of the proof is similar to~\cite{Bar98} (see also \cite[Section 5]{DSW14}), which considers the maximum point of a function, and using the fact that all derivatives are $0$ the result is proved. Here we consider a similar function $f$ defined in Section~\ref{sec:barthefunc}. However, since our problem is more complicated, it is unclear whether we can find a maximum point at which all derivatives are $0$. Instead we will show that there is a point with very small derivatives in Section~\ref{sec:optimumpt}, which is sufficient for our proof of the theorem in Section~\ref{sec:barthepf}.

\subsection{The function and basic properties} \label{sec:barthefunc}

Let $k_1,k_2,\ldots,k_n$ be the dimensions of $V_1,V_2,\ldots,V_n$ respectively and $m=k_1+k_2+\cdots+k_n$. Throughout our proof, we use pairs $(i,j)$ with  $i\in[n]$, $j\in[k_i]$ to denote the element of $[m]$ of position $\sum_{i'<i}k_{i'} + j$. We define a vector $\vec{\gamma}\in\R^m$ as
$$\gamma_{ij}=p_i\qquad\forall i\in[n],j\in[k_i].$$
For every $i\in[n]$, we fix $\{\vec{v}_{i1},\vec{v}_{i2},\ldots,\vec{v}_{ik_i}\}$ to be some basis of $V_i$ (not necessarily orthonormal). A set $I\subseteq[m]$  is called a {\em good basis set} if
$$I=\bigcup_{i\in H}\big\{(i,1),(i,2),\ldots,(i,k_i)\big\}$$
for some $\mathcal{V}$-admissible basis set $H$. We can see that for any good basis set $I$, the set $\{\vec{v}_{ij}:(i,j)\in I\}$ is a basis of $\R^\ell$. For a list of vectors $\vec{a}_1,\vec{a}_2,\ldots,\vec{a}_q$ ($q\in\Z^+$), we use $[\vec{a}_1,\vec{a}_2,\ldots,\vec{a}_q]$ to denote the matrix consisting of columns $\vec{a}_1,\vec{a}_2,\ldots,\vec{a}_q$.

Let $\mathbf{O}(s)$ be the group of $s\times s$ orthogonal matrices. The function $f:\R^m\times \mathbf{O}(k_1)\times \mathbf{O}(k_2)\times\cdots\times \mathbf{O}(k_n)\mapsto\R$ is defined as
$$f(\vec{t},R_1,\ldots,R_n)=\langle\vec{\gamma},\vec{t}\rangle-\ln\det\left(\sum_{i\in[n],j\in[k_i]}e^{t_{ij}}\vec{x}_{ij}\vec{x}_{ij}^T\right),$$
where, for every $i\in[n]$, the vectors $\vec{x}_{ij}$ are given by
$$ [\vec{x}_{i1},\ldots,\vec{x}_{ik_i}]=[\vec{v}_{i1},\ldots,\vec{v}_{ik_i}]R_i.$$
We note that here for every $i\in[n]$, $j\in[k_i]$, $\vec{x}_{ij}$ is a function of $R_i$ and $\{\vec{x}_{i1},\ldots,\vec{x}_{ik_i}\}$ is another basis of $V_i$.

The next lemma shows that the function $f$ is bounded over its domain. The proof is similar to Proposition~3 in \cite{Bar98}. For completeness, we include the proof here.

\begin{lemma} \label{lem:bounded}
There is a constant $C\in\R$ such that $f(\vec{t},R_1,\ldots,R_n)\leq C$ for all $\vec{t}\in\R^m$ and $R_i\in \mathbf{O}(k_i)$ ($i\in[n]$).
\end{lemma}

\begin{proof}
In this proof, we use $\mathcal{F}=\binom{[m]}{\ell}$ to denote the family of all $\ell$-subsets of $[m]$. For a set $I\subseteq[m]$, let $\vec{1}_I\in\{0,1\}^m$ be the indicator vector of $I$, i.e. the $i$-th entry is $1$ iff $i\in I$. By the definition of the vector $\vec{\gamma}$, we can pick $\mu_I\in[0,1]$, $\sum_{I\in\mathcal{F}}\mu_I=1$ so that
$$\vec{\gamma}=\sum_{I\in\mathcal{F}}\mu_I\vec{1}_I,$$
and $\mu_I\neq0$ only when $I$ is a good basis set.

In the proof, we will use the Cauchy-Binet formula which states that for a $\ell\times m$ matrix $A$ and an $m\times\ell$ matrix $B$,
\begin{equation} \label{eqn:cauchybinet}
\det(AB)=\sum_{I\in\mathcal{F}}\det(A_I)\det(B_I),
\end{equation}
where $A_I$ denotes the $\ell\times\ell$ matrix that consists of the subset of $A$'s columns with indices in $I$, and $B_I$ denotes the $\ell\times\ell$ matrix that consists of the subset of $B$'s rows with indices in $I$.

We use $t_I$ to denote the sum of the entries in $\vec{t}$ with indices in $I$, and $L_I$ be the $\ell\times\ell$ submatrix of $[\vec{x}_{11},\ldots,\ldots,\vec{x}_{nk_n}]$ containing only the columns with indices in $I$. We then have
\begin{equation} \label{eqn:gammat}
\langle\vec{\gamma},\vec{t}\rangle=\Big\langle\sum_I\mu_I\vec{1}_I,\vec{t}\Big\rangle=\sum_I\mu_It_I.
\end{equation}

Using equations~(\ref{eqn:cauchybinet}) and~(\ref{eqn:gammat}),
\begin{equation*}\begin{split}
\det\left(\sum_{i\in[n],j\in[k_i]}e^{t_{ij}}\vec{x}_{ij}\vec{x}_{ij}^T\right) & = \det\left(\Big[\vec{x}_{11},\ldots,\vec{x}_{nk_n}\Big]\left[\begin{array}{c}e^{t_{11}}\vec{x}_{11}^T \\ \vdots \\ e^{t_{nk_n}}\vec{x}_{nk_n}^T\end{array}\right]\right) \\
& = \sum_{I\in\mathcal{F}} e^{t_I}\det(L_I)\cdot\det(L_I^T)\hspace{2.01cm}\text{(By~(\ref{eqn:cauchybinet}))} \\
& = \sum_{I\in\mathcal{F}:\mu_I\neq0}\mu_I\left(\frac{e^{t_I}}{\mu_I}\right)\det(L_I)^2+\sum_{I\in\mathcal{F}:\mu_I=0}e^{t_I}\det(L_I)^2 \\
& \geq \prod_{I\in\mathcal{F}:\mu_I\neq0}\left(\frac{e^{t_I}\det(L_I)^2}{\mu_I}\right)^{\mu_I}+0\hspace{0.99cm}\text{(AM-GM inequality)}\\
& = e^{\langle\vec{\gamma},\vec{t}\rangle}\cdot\prod_{I\in\mathcal{F}:\mu_I\neq0}\left(\frac{\det(L_I)^2}{\mu_I}\right)^{\mu_I}\hspace{1cm}\text{(By~(\ref{eqn:gammat}))}.
\end{split}\end{equation*}
Take the logarithm of both sides,
$$f(\vec{t},R_1,\ldots,R_n)=\langle\vec{\gamma},\vec{t}\rangle-\ln\det\left(\sum_{i\in[n],j\in[k_i]}e^{t_{ij}}\vec{x}_{ij}\vec{x}_{ij}^T\right)\leq\sum_{I:\mu_I\neq0}\mu_I\ln\left(\frac{\mu_I}{\det(L_I)^2}\right).$$
The right side is a function of the orthogonal matrices $R_1,R_2,\ldots,R_n$ because $L_I$ is a function of them. We use $\widetilde{f}(R_1,R_2,\ldots,R_n)$ to denote the right side of the above inequality. For $\mu_I\neq0$, $I$ must be a good basis set. Hence $\det(L_I)\neq0$ no matter what the orthogonal matrices $R_1,R_2,\ldots,R_n$ are, and $\widetilde{f}$ is a well-defined continuous function. Since $\widetilde{f}$ is defined on the compact set $\mathbf{O}(k_1)\times \mathbf{O}(k_2)\times\cdots\times \mathbf{O}(k_n)$, it must have a finite upper bound. And that is also an upper bound for the function $f$.
\end{proof}

\subsection{Finding a point with small derivatives} \label{sec:optimumpt}

We first define some notations. Let
$$X=\sum\limits_{i\in[n],j\in[k_i]} e^{t_{ij}}\vec{x}_{ij}\vec{x}_{ij}^T$$
be a matrix valued function of $\vec{t},R_1,R_2,\ldots,R_n$. Then
$$f(\vec{t},R_1,\ldots,R_n)=\langle\vec{\gamma},\vec{t}\rangle-\ln\det(X).$$
Note that $X$ is always a positive definite matrix, since for any $\vec{w}\neq\vec{0}$,
$$\vec{w}^TX\vec{w}=\sum_{i\in[n],j\in[k_i]}e^{t_{ij}}\langle\vec{x}_{ij},\vec{w}\rangle^2>0,$$
when $\vec{x}_{11},\ldots,\ldots,\vec{x}_{nk_n}$ span the entire space (implied by $V_1+V_2+\cdots+V_n=\R^\ell$). Define $M$ to be the $\ell\times\ell$ full rank matrix satisfying $M^TM=X^{-1}$. We note that $M$ is also a function of $\vec{t},R_1,R_2,\ldots,R_n$.

In a later part of the proof we will show that the linear map defined by $M$ satisfies the requirement in Theorem~\ref{thm:barthe} when $\vec{t}$, $R_1,R_2,\ldots,R_n$ take appropriate values. We first find an appropriate value of $(R_1,R_2,\ldots,R_n)=(R_1^*(\vec{t}),R_2^*(\vec{t}),\ldots,R_n^*(\vec{t}))$ for every $\vec{t}\in\R^m$, and then find some $\vec{t}^*$ with specific properties.

\begin{lemma} \label{lem:optr}
For every $\vec{t}\in\R^m$, there exists $\big(R_1^*(\vec{t}),R_2^*(\vec{t}),\ldots,R_n^*(\vec{t})\big)$ satisfying
\begin{enumerate}
\item $f\big(\vec{t},R_1^*(\vec{t}),R_2^*(\vec{t}),\ldots,R_n^*(\vec{t})\big)=\max_{R_1,R_2,\ldots,R_n}\big\{f(\vec{t},R_1,R_2,\ldots,R_n)\big\}$.
\item For every $i\in[n]$, if $t_{ij}=t_{ij'}$ for some $j\neq j'\in[k_i]$, then
$$\langle M\vec{x}_{ij},M\vec{x}_{ij'}\rangle=0,$$
where $[\vec{x}_{i1},\ldots,\vec{x}_{ik_i}]=[\vec{v}_{i1},\ldots,\vec{v}_{ik_i}]R_i^*(\vec{t})$.
\end{enumerate}
\end{lemma}

\begin{proof}
The first condition can be satisfied by the compactness of $\mathbf{O}(k_1)\times \mathbf{O}(k_2)\times\cdots\times \mathbf{O}(k_n)$. We will show how to change $(R_1^*(\vec{t}),R_2^*(\vec{t}),\ldots,R_n^*(\vec{t}))$, which already satisfies the first condition, so that it satisfies the second condition while preserving the first condition.

Fix an $i\in[n]$ and partition the indices of $(t_{i1},t_{i2},\ldots,t_{ik_i})$ into equivalence classes $J_1,J_2,\ldots,J_b\subseteq[k_i]$ such that for $j,j'$ in the same class $t_{ij}=t_{ij'}$ and for $j,j'$ in different classes $t_{ij}\neq t_{ij'}$. We use $t_{J_r}$ to denote the value of $t_{ij}$ for $j\in J_r$, and $L_{J_r}$ to denote the matrix consisting of all columns $\vec{x}_{ij}$ with $j\in J_r$. The terms in $X$ that depend on $R_i$ are
$$\sum_{r\in[b]}\left(e^{t_{J_r}}\sum_{j\in J_r}\vec{x}_{ij}\vec{x}_{ij}^T\right)=\sum_{r\in[b]}\left(e^{t_{J_r}}\cdot L_{J_r}L_{J_r}^T\right)=\sum_{r\in[b]}\left(e^{t_{J_r}}\cdot L_{J_r}Q_rQ_r^TL_{J_r}^T\right),$$
where $Q_r$ can be taken to be any $|J_r|\times|J_r|$ orthogonal matrix. This means that if we change $R_i^*(\vec{t})$ to $R_i^*(\vec{t})\diag(Q_1,\ldots,Q_b)$ (here $\diag(Q_1,\ldots,Q_b)$ denotes the matrix in which the submatrix with row and column indices $J_r$ is $Q_r$), or equivalently change $L_{J_r}$ to $L_{J_r}Q_r$ for every $r\in[b]$, the matrix $X$ does not change, hence $M$ and $f$ do not change, and the first condition is preserved as $f$ is still the maximum for the fixed $\vec{t}$.

For every $r\in[b]$, we can find a $Q_r$ such that the columns of $ML_{J_r}Q_r$ are orthogonal (consider the singular value decomposition of $ML_{J_r}$). Change $R_i^*(\vec{t})$ to $R_i^*(\vec{t})\diag(Q_1,\ldots,Q_b)$ and the second condition is satisfied while preserving the first condition. Doing this for every $i$ we can obtain an $(R_1^*(\vec{t}),R_2^*(\vec{t}),\ldots,R_n^*(\vec{t}))$ satisfying both conditions.
\end{proof}

From now on we use $R_1^*(\vec{t}),R_2^*(\vec{t}),\ldots,R_n^*(\vec{t})$ to denote the matrices satisfying the conditions in Lemma~\ref{lem:optr}.

\begin{lemma} \label{lem:derivative}
For any $\varepsilon>0$, there exists $\vec{t}^*\in\R^m$ such that for every $i\in[n],j\in[k_i]$.
$$\left|\frac{\partial f}{\partial t_{ij}}\Big(\vec{t}^*,R_1^*(\vec{t}^*),R_2^*(\vec{t}^*),\ldots,R_n^*(\vec{t}^*)\Big)\right|\leq\varepsilon.$$
\end{lemma}

This lemma follows immediately from the following more general lemma.

\begin{lemma} \label{lem:smallderiv}
Let $\mathcal{A}\subseteq\R^{h}$ ($h\in\Z^+$) be a compact set. Let $f:\R^m\times\mathcal{A}\mapsto\R$ and $y^*:\R^m\mapsto\mathcal{A}$ be functions satisfying the following properties:
\begin{enumerate}
\item $f(\vec{x},y)$ is bounded and continuous on $\R^m\times\mathcal{A}$.
\item For every $\vec{x}\in\R^m$, $f(\vec{x},y^*(\vec{x}))=\max_{y\in\mathcal{A}}\{f(\vec{x},y)\}$.
\item For every fixed $y\in\mathcal{A}$, $f(\vec{x},y)$ as a function of $\vec{x}$ is differentiable on $\R^m$.
\end{enumerate}
Then, for every $\varepsilon>0$, there exists an $\vec{x}^*\in\R^m$ such that for every $i\in[m]$,
$$\left|\frac{\partial f}{\partial x_i}\Big(\vec{x}^*,y^*(\vec{x}^*)\Big)\right|\leq\varepsilon.$$
\end{lemma}

\begin{proof}
We denote by $f^*(\vec{x})=f\big(\vec{x},y^*(\vec{x})\big)$. For the sake of contradiction, assume that for any $\vec{x}\in\R^m$, there is an index $i\in[m]$ such that
\begin{equation} \label{eqn:bigderivative}
\left|\frac{\partial f}{\partial x_i}\Big(\vec{x},y^*(\vec{x})\Big)\right|>\varepsilon.
\end{equation}
In particular, there is a derivative greater than $\varepsilon$ at $\vec{x}=\vec{0}$. Therefore there exists an $\vec{x}_0\neq\vec{0}$ such that
$$f^*(\vec{x}_0)-f^*(\vec{0})\geq f\big(\vec{x}_0,y^*(\vec{0})\big)-f\big(\vec{0},y^*(\vec{0})\big)\geq0.9\varepsilon\cdot\|\vec{x}_0-\vec{0}\|=0.9\varepsilon\cdot\|\vec{x}_0\|.$$
Define
$$g(\vec{x},y)=f(\vec{x},y)-f^*(\vec{0})-0.9\varepsilon\cdot\|\vec{x}\|,$$
and
$$\mathcal{G}=\big\{(\vec{x},y)\in\R^m\times\mathcal{A}: g(\vec{x},y)\geq0\big\}=g^{-1}\big([0,+\infty)\big).$$
We can see $\mathcal{G}\neq\emptyset$ by $\vec{x}_0\in\mathcal{G}$. By Property~1, $f(\vec{x},y)$ is bounded and any $(\vec{x},y)$ with sufficiently large $\|\vec{x}\|$ cannot be in $\mathcal{G}$. Hence $\mathcal{G}$ is bounded. Since $g(\vec{x},y)$ is a continuous function by Property~2, the set $\mathcal{G}=g^{-1}\big([0,+\infty)\big)$ must be closed. Therefore $\mathcal{G}$ is compact. Thus we can find
$$Z=\max_{(\vec{x},y)\in\mathcal{G}}\big\{\|\vec{x}\|\big\}.$$
Pick $(\vec{x}_1,y_1)\in\mathcal{G}$ with $\|\vec{x}_1\|=Z$. The point $(\vec{x}_1,y_1)$ is in the compact set
$$\mathcal{B}_Z=\big\{\vec{x}\in\R^m:\|\vec{x}\|=Z\big\}\times\mathcal{A}.$$
Let $\big(\vec{x}_1^*,y^*(\vec{x}_1^*)\big)\in\mathcal{B}_Z$ be any point where $f$ is maximized over $\mathcal{B}_Z$. By $(\vec{x}_1,y_1)\in\mathcal{G}$, we have
\begin{equation} \label{eqn:xpping}
f^*(\vec{x}_1^*)-f^*(\vec{0})\geq f(\vec{x}_1,y_1)-f^*(\vec{0})\geq0.9\varepsilon\cdot\|\vec{x}_1\|=0.9\varepsilon\cdot\|\vec{x}_1^*\|.
\end{equation}
By~(\ref{eqn:bigderivative}), there must be an $\vec{x}_2\neq\vec{x}_1^*$ such that
\begin{equation} \label{eqn:xpppdef}
f^*(\vec{x}_2)-f^*(\vec{x}_1^*)\geq f\big(\vec{x}_2,y^*(\vec{x}_1^*)\big)-f\big(\vec{x}_1^*,y^*(\vec{x}_1^*)\big)\geq0.9\varepsilon\cdot\|\vec{x}_2-\vec{x}_1^*\|.
\end{equation}
Note that $f^*(\vec{x}_2)$ is strictly greater than $f^*(\vec{x}_1^*)$. By the maximality of $f\big(\vec{x}_1^*,y^*(\vec{x}_1^*)\big)$ on $\mathcal{B}_Z$, we can see $\|\vec{x}_2\|\neq Z$. There are two cases:
\begin{enumerate}
\item $\|\vec{x}_2\|<Z$. This implies that the maximum value of $f$ over
$$\mathcal{B}_{\leq Z}=\big\{\vec{x}\in\R^m:\|\vec{x}\|\leq Z\big\}\times\mathcal{A}$$
is at least $f^*(\vec{x}_2)>f^*(\vec{x}_1^*)$. Say $f$ archives the maximum value over $\mathcal{B}_{\leq Z}$ at $(\vec{x},y)=\big(\vec{x}_3,y^*(\vec{x}_3)\big)$. Then we have $\|\vec{x}_3\|<Z$ by the maximality of $f\big(\vec{x}_1^*,y^*(\vec{x}_1^*)\big)$ on $\mathcal{B}_Z$. And $\vec{x}=\vec{x}_3$ must be a local maximum of $f\big(\vec{x},y^*(\vec{x}_3)\big)$ with $y=y^*(\vec{x}_3)$ fixed. Therefore
$$\frac{\partial f}{\partial x_i}\Big(\vec{x}_3,y^*(\vec{x}_3)\Big)=0\quad\forall i\in[m],$$
violating~(\ref{eqn:bigderivative}).
\item $\|\vec{x}_2\|>Z$. By~(\ref{eqn:xpping}) and~(\ref{eqn:xpppdef}), we have
\begin{equation*}\begin{split}
f^*(\vec{x}_2)-f^*(\vec{0}) & =f^*(\vec{x}_2)-f^*(\vec{x}_1^*)+f^*(\vec{x}_1^*)-f^*(\vec{0}) \\
&\geq0.9\varepsilon\cdot\|\vec{x}_2-\vec{x}_1^*\|+0.9\varepsilon\cdot\|\vec{x}_1^*\| \\
&\geq0.9\varepsilon\cdot\|\vec{x}_2\|.
\end{split}\end{equation*}
Therefore $\big(\vec{x}_2,y^*(\vec{x}_2)\big)\in\mathcal{G}$. By the definition of $Z$, there should be $\|\vec{x}_2\|\leq Z$, contradiction.
\end{enumerate}
Thus the lemma is proved.
\end{proof}

\subsection{Proof of Theorem~\ref{thm:barthe}} \label{sec:barthepf}

We apply Lemma~\ref{lem:derivative} with $\varepsilon'=\varepsilon/m$ and obtain a $\vec{t}^*$. In the remaining proof we will use $X$, $M$ and $\vec{x}_{ij}$ ($i\in[n],j\in[k_i]$) to denote their values when $\vec{t}=\vec{t}^*$ and $R_i=R_i^*(\vec{t}^*)$ ($i\in[n]$).
\begin{lemma}
$\langle M\vec{x}_{ij},M\vec{x}_{ij'}\rangle=0$ for every $i\in[n]$ and $j\neq j'\in[k_i]$.
\end{lemma}

\begin{proof}
We fix $i_0\in[n], j_0\neq j_0'\in[k_{i_0}]$ and prove $\langle M\vec{x}_{i_0j_0},M\vec{x}_{i_0j_0'}\rangle=0$. If $t_{i_0j_0}^*=t_{i_0j_0'}^*$, this is guaranteed by Lemma~\ref{lem:optr}. We only consider the case that $t_{i_0j_0}^*\neq t_{i_0j_0'}^*$.

Let $\theta\in\R$ be a variable, and define $\vec{x}_{ij}'$ for $i\in[n]$, $j\in[k_i]$ as follows.
$$\vec{x}_{ij}'=\begin{cases}
\cos\theta\cdot\vec{x}_{i_0j_0}-\sin\theta\cdot\vec{x}_{i_0j_0'} & \quad(i,j)=(i_0,j_0), \\
\sin\theta\cdot\vec{x}_{i_0j_0}+\cos\theta\cdot\vec{x}_{i_0j_0'} & \quad(i,j)=(i_0,j_0'), \\
\vec{x}_{ij} & \quad\text{otherwise}.
\end{cases}$$
We consider the following function $h:\R\mapsto\R$,
$$h(\theta)=\langle\vec{\gamma},\vec{t}^*\rangle-\ln\det\left(\sum_{i\in[n],j\in[k_i]}e^{t_{ij}^*}\vec{x}_{ij}'{\vec{x}_{ij}'}^T\right).$$

\begin{claim}
$h(\theta)$ has a maximum at $\theta=0$.
\end{claim}

\begin{proof}
Let $R(\theta)$ be the $k_{i_0}\times k_{i_0}$ orthogonal matrix obtained from the identity matrix by changing the $(j_0,j_0)$, $(j_0',j_0')$ entries to $\cos\theta$, the $(j_0,j_0')$ entry to $\sin\theta$, and the $(j_0',j_0)$ entry to $-\sin\theta$. We can see $R(0)$ is the identity matrix and
$$[\vec{x}_{i_01}',\ldots,\vec{x}_{i_0k_{i_0}}']=[\vec{x}_{i_01},\ldots,\vec{x}_{i_0k_{i_0}}]R(\theta).$$
Therefore for all $\theta\in\R$.
\begin{equation*}\begin{split}
h(\theta) & = f\Big(\vec{t}^*,R_1^*(\vec{t}^*),\ldots,R_{i_0-1}^*(\vec{t}^*),R_{i_0}^*(\vec{t}^*)\cdot R(\theta),R_{i_0+1}^*(\vec{t}^*),\ldots,R_n^*(\vec{t}^*)\Big) \\
& \leq f\Big(\vec{t}^*,R_1^*(\vec{t}^*),\ldots,R_{i_0-1}^*(\vec{t}^*),R_{i_0}^*(\vec{t}^*),R_{i_0+1}^*(\vec{t}^*),\ldots,R_n^*(\vec{t}^*)\Big) \\
& = h(0).
\end{split}\end{equation*}
Thus the claim is proved.
\end{proof}

Using $\frac{d}{ds}\ln\det(A)=\tr(A^{-1}\frac{d}{ds}A)$ for invertible matrix $A$ (Theorem~4 in~\cite[Chapter~9]{Lax07}), we can calculate the derivative of $h$.
\begin{equation*}\begin{split}
\frac{d h}{d\theta}(0) = & -\tr\Big[X^{-1}\Big(e^{t_{i_0j_0}^*}\left.\frac{d}{d\theta}\right|_{\theta=0}\vec{x}_{i_0j_0}'{\vec{x}_{i_0j_0}'}^T+e^{t_{i_0j_0'}^*}\left.\frac{d}{d\theta}\right|_{\theta=0}\vec{x}_{i_0j_0'}'{\vec{x}_{i_0j_0'}'}^T\Big)\Big] \\
= & -\tr\Big[X^{-1}\Big(e^{t_{i_0j_0}^*}\left.\frac{d}{d\theta}\right|_{\theta=0}(\cos\theta\cdot\vec{x}_{i_0j_0}-\sin\theta\cdot\vec{x}_{i_0j_0'})(\cos\theta\cdot\vec{x}_{i_0j_0}-\sin\theta\cdot\vec{x}_{i_0j_0'})^T \\
& \hspace{1.4cm} +e^{t_{i_0j_0'}^*}\left.\frac{d}{d\theta}\right|_{\theta=0}(\sin\theta\cdot\vec{x}_{i_0j_0}+\cos\theta\cdot\vec{x}_{i_0j_0'})(\sin\theta\cdot\vec{x}_{i_0j_0}+\cos\theta\cdot\vec{x}_{i_0j_0'})^T\Big)\Big] \\
= & -e^{t_{i_0j_0}^*}\tr\Big[\left.\frac{d}{d\theta}\right|_{\theta=0}(\cos\theta\cdot M\vec{x}_{i_0j_0}-\sin\theta\cdot M\vec{x}_{i_0j_0'})(\cos\theta\cdot M\vec{x}_{i_0j_0}-\sin\theta\cdot M\vec{x}_{i_0j_0'})^T\Big] \\
& -e^{t_{i_0j_0'}^*}\tr\Big[\left.\frac{d}{d\theta}\right|_{\theta=0}(\sin\theta\cdot M\vec{x}_{i_0j_0}+\cos\theta\cdot M\vec{x}_{i_0j_0'})(\sin\theta\cdot M\vec{x}_{i_0j_0}+\cos\theta\cdot M\vec{x}_{i_0j_0'})^T\Big] \\
= & -e^{t_{i_0j_0}^*}\big[-2\cdot\langle M\vec{x}_{i_0j_0},M\vec{x}_{i_0j_0'}\rangle\big]-e^{t_{i_0j_0'}^*}\big[2\cdot\langle M\vec{x}_{i_0j_0},M\vec{x}_{i_0j_0'}\rangle\big] \\
= & 2(e^{t_{i_0j_0}^*}-e^{t_{i_0j_0'}^*})\cdot\langle M\vec{x}_{i_0j_0},M\vec{x}_{i_0j_0'}\rangle.
\end{split}\end{equation*}
Since $h(0)$ is the maximum, we have $\frac{d h}{d\theta}(0)=0$. By $t_{i_0j_0}^*\neq t_{i_0j_0'}^*$, the above equation implies $\langle M\vec{x}_{i_0j_0},M\vec{x}_{i_0j_0'}\rangle=0$.
\end{proof}

Finally we are able to prove Theorem~\ref{thm:barthe}.

\begin{proof}[Proof of Theorem~\ref{thm:barthe}]
With a slight abuse of notation, we also use $M$ to denote the linear map defined by the matrix $M$. We show that $M$ satisfies the requirement in Theorem~\ref{thm:barthe}. Let $\vec{u}_{ij}=M\vec{x}_{ij}/\|M\vec{x}_{ij}\|$ ($i\in[n]$, $j\in[k_i]$). Then $\{\vec{u}_{i1},\vec{u}_{i2},\ldots,\vec{u}_{ik_i}\}$ is an orthonormal basis of $M(V_i)$, and
\begin{equation} \label{eqn:projmatrix}
\Proj_{M(V_i)}=[\vec{u}_{i1},\vec{u}_{i2},\ldots,\vec{u}_{ik_i}]\left[\begin{array}{c}\vec{u}_{i1}^T \\ \vdots \\ \vec{u}_{ik_i}^T\end{array}\right]=\sum_{j=1}^{k_i}\vec{u}_{ij}\vec{u}_{ij}^T.
\end{equation}
We define
$$\varepsilon_{ij}=\frac{\partial f}{\partial t_{ij}}\Big(\vec{t}^*,R_1^*(\vec{t}^*),R_2^*(\vec{t}^*),\ldots,R_n^*(\vec{t}^*)\Big)\in[-\frac{\varepsilon}{m},\frac{\varepsilon}{m}].$$
Note that $\frac{d}{ds}\ln\det(A)=\tr(A^{-1}\frac{d}{ds}A)$ for invertible matrix $A$ (Theorem~4 in~\cite[Chapter~9]{Lax07}). We have
$$\varepsilon_{ij}= p_i-\tr\left(X^{-1}e^{t_{ij}^*}\vec{x}_{ij}\vec{x}_{ij}^T\right)=p_i-e^{t_{ij}^*}\cdot\tr\left(M\vec{x}_{ij}\vec{x}_{ij}^TM^T\right)=p_i-e^{t_{ij}^*}\cdot\|M\vec{x}_{ij}\|^2.$$
By the definition of $X$ and $M$,
$$M^{-1}(M^T)^{-1}=X=\sum_{i\in[n],j\in[k_i]}e^{t_{ij}^*}\vec{x}_{ij}\vec{x}_{ij}^T\quad\Longrightarrow\quad\sum_{i\in[n],j\in[k_i]}e^{t_{ij}^*}(M\vec{x}_{ij})(M\vec{x}_{ij})^T=I_{\ell\times\ell}.$$
Therefore
$$\sum_{i\in[n],j\in[k_i]}(p_i-\varepsilon_{ij})\vec{u}_{ij}\vec{u}_{ij}^T=\sum_{i\in[n],j\in[k_i]}e^{t_{ij}^*}\|M\vec{x}_{ij}\|^2\left(\frac{M\vec{x}_{ij}}{\|M\vec{x}_{ij}\|}\right)\left(\frac{M\vec{x}_{ij}}{\|M\vec{x}_{ij}\|}\right)^T=I_{\ell\times\ell}.$$
By~(\ref{eqn:projmatrix}),
$$\Big\|\sum_{i=1}^np_i\Proj_{M(V_i)}-I_{\ell\times\ell}\Big\|=\Big\|\sum_{i\in[n],j\in[k_i]}\varepsilon_{ij}\vec{u}_{ij}\vec{u}_{ij}^T\Big\|\leq\frac{\varepsilon}{m}\sum_{i\in[n],j\in[k_i]}\|\vec{u}_{ij}\vec{u}_{ij}^T\|\leq\varepsilon.$$
Thus Theorem~\ref{thm:barthe} is proved.
\end{proof}

\subsection{A convenient form of Theorem~\ref{thm:barthe}}

We give Theorem~\ref{thm:barthec} which is implied by Theorem~\ref{thm:barthe} and is the form that will be used in our proof.  Before stating the theorem, we need to define {\em admissible sets} and {\em admissible vectors} as Definition~\ref{def:adsv}, which have weaker requirements than admissible basis sets and admissible basis vectors (Definition~\ref{def:adbasic}) as they are not required to span the entire arrangement.

\begin{definition}[admissible  set, admissible  vector] \label{def:adsv}
Given a list of vector spaces $\mathcal{V}=(V_1,V_2,\ldots,V_n)$ ($V_i\subseteq\R^\ell$), a set $H\subseteq [n]$ is called a {\em $\mathcal{V}$-admissible  set} if
 $\dim(\sum_{i\in H}V_i)=\sum_{i\in H}\dim(V_i),$ i.e. if every space with index in $H$ has intersection $\{\vec{0}\}$ with the span of the other spaces with indices in $H$. A {\em $\mathcal{V}$-admissible vector} is any indicator vector $\vec{1}_H$ of some $\mathcal{V}$-admissible  set $H$. 
\end{definition}

\begin{theorem} \label{thm:barthec}
Given a list of vector spaces $\mathcal{V}=(V_1,V_2,\ldots,V_n)$ ($V_i\subseteq\R^\ell$) and a vector $\vec{p}\in\R^n$ in the convex hull of all $\mathcal{V}$-admissible vectors. Then there exists an invertible linear map $M:\R^\ell\mapsto\R^\ell$ such that for any unit vector $\vec{w}\in\R^\ell$,
$$\sum_{i=1}^np_i\|\Proj_{M(V_i)}(\vec{w})\|^2\leq2,$$
where $\Proj_{M(V_i)}(\vec{w})$ is the projection of $\vec{w}$ onto $M(V_i)$.
\end{theorem}

Note that with a slight abuse of notation we use $\Proj_{M(V_i)}$ to denote both the projection matrix and the projection map.

\begin{proof}
We use $V$ to denote $V_1+V_2+\cdots+V_n$. Let $d=\dim(V)$ and $\{\vec{b}_1,\vec{b}_2,\ldots,\vec{b}_d\}$ be some orthonormal basis of $V$. We construct $(\mathcal{V}',\vec{p}')$ satisfying the conditions in Theorem~\ref{thm:barthe} in the following 2 steps.
\begin{enumerate}
\item In this step, we construct $\widetilde{\mathcal{V}}$ and $\vec{p}'$ so that $\vec{p}'$ is in the convex hull of all $\widetilde{\mathcal{V}}$-admissible basis vectors.  Define $V_{n+1}=\spn\{\vec{b}_1\}$, $V_{n+2}=\spn\{\vec{b}_2\}$, \ldots, $V_{n+d}=\spn\{\vec{b}_d\}$ and
$$\widetilde{\mathcal{V}}=(V_1,V_2,\ldots,V_n,V_{n+1},V_{n+2},\ldots,V_{n+d}).$$
For every $\mathcal{V}$-admissible set $H\subseteq[n]$, we can see that $H$ is also $\widetilde{\mathcal{V}}$-admissible, and there is a subset $G\subseteq\{n+1,n+2,\ldots,n+d\}$ such that $H'=H\cup G$ is a $\widetilde{\mathcal{V}}$-admissible basis set. Assume
$$\vec{p}=\sum_{\mathcal{V}\text{-admissible }H}\mu_H\vec{1}_H,$$
where $\mu_H\in[0,1]$ and $\sum\mu_H=1$. We define
$$\vec{p}'=\sum_{\mathcal{V}\text{-admissible }H}\mu_H\vec{1}_{H'},$$
where $H'$ is the $\widetilde{\mathcal{V}}$-admissible basis set extended from $H$ as above. We can see that $\vec{p}$ is a prefix of $\vec{p}'$, and $\vec{p}'$ is in the convex hull of all $\widetilde{\mathcal{V}}$-admissible basis vectors.
\item In this step, we construct $\mathcal{V}'$ based on $\widetilde{\mathcal{V}}$ so that the vector spaces span the entire Euclidean space. We find an isomorphism linear map $P:V\mapsto\R^d$ such that $P(\vec{b}_i)=\vec{e}_i$ for $i\in[d]$, where $\{\vec{e}_1,\vec{e}_2,\ldots,\vec{e}_d\}$ is the standard basis of $\R^d$. Define
$$\mathcal{V}'=\big(V_1',V_2',\ldots,V_{n+d}'\big)=\big(P(V_1),P(V_2),\ldots,P(V_{n+d})\big).$$

We can see that $V_1'+V_2'+\cdots+V_{n+d}'=\R^d$ and $\vec{p}'$ is in the convex hull of all $\mathcal{V}'$-admissible basis vectors. Hence $(\mathcal{V}',\vec{p}')$ satisfy the conditions in Theorem~\ref{thm:barthe}.
\end{enumerate}

Apply Theorem~\ref{thm:barthe} on $(\mathcal{V}',\vec{p}')$ with $\varepsilon=1$. There exist an invertible linear map $M':\R^d\mapsto\R^d$ such that
$$\Big\|\sum_{i=1}^{n+d}p_i'\Proj_{M'(V_i')}-I_{d\times d}\Big\|\leq1.$$
For every unit vector $\vec{w}'\in\R^d$, we have
\begin{equation*}\begin{split}
&1\geq\vec{w}^T\left(\sum_{i=1}^{n+d}p_i'\Proj_{M'(V_i')}-I_{d\times d}\right)\vec{w}=\sum_{i=1}^{n+d}p_i'\|\Proj_{M'(V_i')}(\vec{w})\|^2-1, \\
\Longrightarrow\quad&\sum_{i=1}^{n+d}p_i'\|\Proj_{M'(V_i')}(\vec{w})\|^2\leq2.
\end{split}\end{equation*}
Note that the linear map $P$ defined in Step~2 only changes orthonormal basis. We find an invertible linear map $M:\R^\ell\mapsto\R^\ell$ such that $M(\vec{v})=P^{-1}(M'(P(\vec{v})))$ for every $\vec{v}\in V$. Then for every unit vector $\vec{w}\in V$,
$$\sum_{i=1}^{n+d}p_i'\|\Proj_{M(V_i)}(\vec{w})\|^2\leq2.$$
It is easy to see that the same inequality holds for every unit vector $\vec{w}\in\R^\ell$. Recall that in Step~1, $\vec{p}$ is a prefix of $\vec{p}'$. The theorem is proved because the above inequality is stronger than the required.
\end{proof}

\section{Proof of the main Theorem} \label{sec:main}

 Theorem~\ref{thm:main} will follow from the following theorem using a simple recursive argument.
\begin{theorem} \label{thm:main_}
Let $\mathcal{V}=(V_1,V_2,\ldots,V_n)$ ($V_i\in\R^\ell$) be a list of $k$-bounded vector spaces with an $(\alpha,\delta)$-system and $d=\dim(V_1+V_2+\cdots+V_n)$, then for any $\beta\in(0,1)$, at least one of these two cases holds:
\begin{enumerate}
\item $d\leq400\alpha k^3/(\beta\delta)$,
\item There is a sublist of $q\geq\delta n/(20\alpha)$ spaces $(V_{i_1},V_{i_2},\ldots,V_{i_q})$ such that there are nonzero vectors $\vec{z}_1\in V_{i_1},\vec{z}_2\in V_{i_2},\ldots,\vec{z}_q\in V_{i_q}$ with 
$$\dim(\vec{z}_1,\vec{z}_2,\ldots,\vec{z}_q)\leq\beta d.$$
\end{enumerate}
\end{theorem}

\begin{proof}
Initially let $\mathcal{V}^{(0)}=(V_1^{(0)},V_2^{(0)},\ldots,V_{n_0}^{(0)})=\mathcal{V}$, $\delta_0=\delta$ and $d_0=d$, where $n_0=n$ and $V_i^{(0)}=V_i$.

Starting with $t=0$, $\mathcal{V}^{(t)}=(V_1^{(t)},V_2^{(t)},\ldots,V_{n_t}^{(t)})$ is a list of $k$-bounded vectors spaces with an $(\alpha,\delta_t)$-system and $d_t=\dim(V_1^{(t)}+V_2^{(t)}+\cdots+V_{n_t}^{(t)})$. We apply Theorem~\ref{thm:main_} on $\mathcal{V}^{(t)}$.
\begin{itemize}
\item If the first case of Theorem~\ref{thm:main_} holds, i.e. $d_t\leq400\alpha k^3/(\beta\delta_t)$, terminate.
\item If the second case of Theorem~\ref{thm:main_} holds, i.e. there exist $\vec{z}_1,\vec{z}_2,\ldots,\vec{z}_q$ from $q\geq\delta_tn_t/(20\alpha)$ spaces such that $\dim(\vec{z}_1,\vec{z}_2,\ldots,\vec{z}_q)\leq\beta d$.

We find a linear map $P:\R^\ell\mapsto\R^\ell$ whose kernel equals $\spn\{\vec{z}_1,\vec{z}_2,\ldots,\vec{z}_q\}$. Define 
$$\mathcal{V}^{(t+1)}=(V_1^{(t+1)},V_2^{(t+1)},\ldots,V_{n_{t+1}}^{(t+1)})$$
as the list of nonzero spaces in $P(V_1^{(t)}),P(V_1^{(t)}),\ldots,P(V_{n_t}^{(t)})$. By Corollary~\ref{cor:remove}, $\mathcal{V}^{(t+1)}$ has an $(\alpha,\delta_{t+1})$-system for $\delta_{t+1}=\delta_tn_t/n_{t+1}$.

Let $t\leftarrow t+1$ and repeat the procedure.
\end{itemize}

In the above procedure, we can see $\delta_tn_t=\delta_{t-1}n_{t-1}=\cdots=\delta_0n_0=\delta n$. Note that each step we map vectors from $q\geq\delta_tn_t/(20\alpha)=\delta n/(20\alpha)$ spaces to $\vec{0}$, hence
$$\dim V_1^{(t+1)}+\dim V_2^{(t+1)}+\cdots+\dim V_{n_{t+1}}^{(t+1)}\leq\dim V_1^{(t)}+\dim V_2^{(t)}+\cdots+\dim V_{n_t}^{(t)}-\frac{\delta n}{20\alpha}.$$
Since initially $\dim V_1+\dim V_2+\cdots+\dim V_n\leq kn$, we must terminate after at most
$$\frac{kn}{\delta n/(20\alpha)}=\frac{20\alpha k}{\delta}$$ steps.

At the $t$-th step,
$$d_t\geq(1-\beta)d_{t-1}\geq\cdots\geq(1-\beta)^td_0=(1-\beta)^td.$$
And if the $t$-th step is the last step we have $d_t\leq400\alpha k^3/(\beta\delta_t)\leq400\alpha k^3/(\beta\delta)$ by $\delta_t\geq\delta$ (implied by $\delta_tn_t=\delta n$ and $n_t\leq n$). Therefore
$$d\leq\left(\frac{1}{1-\beta}\right)^{20\alpha k/\delta}\cdot\frac{400\alpha k^3}{\beta\delta}.$$
We assign $\beta=\min\{1/2,\delta/(\alpha k)\}$. It is easy to verify that $1/(1-\beta)^{\alpha k/\delta}\leq4$ in both cases $\delta/(\alpha k)<1/2$ and $\delta/(\alpha k)\geq1/2$. Therefore
$$d\leq4^{20}\cdot\frac{400\alpha k^3}{\beta\delta}=O(\alpha^2k^4/\delta^2),$$
and Theorem~\ref{thm:main} is proved.
\end{proof}

\subsection{Proof of Theorem~\ref{thm:main_} -- a special case}

In this subsection, we consider the case that all vector spaces are `well separated'.

\begin{definition}
Two vector spaces $V,V'\subseteq\R^\ell$ are {\em $\tau$-separated} if $|\langle\vec{u},\vec{u}'\rangle|\leq1-\tau$ for any two unit vectors $\vec{u}\in V$ and $\vec{u}'\in V'$.
\end{definition}

We will use the following two simple lemmas about $\tau$-separated spaces.
\begin{lemma} \label{lem:smallcoef}
Given two vector spaces $V,V'\subseteq\R^\ell$ that are $\tau$-separated and let $B=\{\vec{u}_1,\vec{u}_2,\ldots,\vec{u}_{k_1}\}$ and $B'=\{\vec{u}_1',\vec{u}_2',\ldots,\vec{u}_{k_2}'\}$ be orthonormal bases for $V,V'$ respectively. For any unit vector $\vec{u}\in V+V'$, if we write $\vec{u}$ as
$$\vec{u}=\lambda_1\vec{u}_1+\lambda_2\vec{u}_2+\cdots+\lambda_{k_1}\vec{u}_{k_1}+\mu_1\vec{u}_1'+\mu_2\vec{u}_2'+\cdots+\mu_{k_2}\vec{u}_{k_2}',$$
then the coefficients satisfy
$\lambda_1^2+\lambda_2^2+\cdots+\lambda_{k_1}^2+\mu_1^2+\mu_2^2+\cdots+\mu_{k_2}^2\leq\frac{1}{\tau}.$
\end{lemma}

\begin{proof}
Let $\vec{v}=\lambda_1\vec{u}_1+\lambda_2\vec{u}_2+\cdots+\lambda_{k_1}\vec{u}_{k_1}$ and $\vec{w}=\mu_1\vec{u}_1'+\mu_2\vec{u}_2'+\cdots+\mu_{k_2}\vec{u}_{k_2}'$. We have
\begin{equation*}\begin{split}
1&=\|\vec{u}\|^2=\|\vec{v}+\vec{w}\|^2=\|\vec{v}\|^2+\|\vec{w}\|^2+2\langle\vec{v},\vec{w}\rangle\geq\|\vec{v}\|^2+\|\vec{w}\|^2-2(1-\tau)\|\vec{v}\|\|\vec{w}\| \\
&\geq\tau(\|\vec{v}\|^2+\|\vec{w}\|^2) \\
&=\tau(\lambda_1^2+\lambda_2^2+\cdots+\lambda_{k_1}^2+\mu_1^2+\mu_2^2+\cdots+\mu_{k_2}^2).\qedhere
\end{split}\end{equation*}
\end{proof}

\begin{lemma} \label{lem:spacetobasis}
Given two vector spaces $V,V'\subseteq\R^\ell$ and let $B=\{\vec{u}_1,\vec{u}_2,\ldots,\vec{u}_{k_1}\}$ be an orthonormal basis of $V$. If $V$ and $V'$ are not $\tau$-separated, there must exist $j\in[k_1]$ such that $\|\Proj_{V'}(\vec{u}_j)\|^2\geq(1-\tau)^2/k_1$, where $\Proj_{V'}(\vec{u}_j)$ is the projection of $\vec{u}_j$ onto $V'$.
\end{lemma}

\begin{proof}
Let $\vec{u}\in V$, $\vec{u}'\in V'$ be unit vectors such that $|\langle\vec{u},\vec{u}'\rangle|>1-\tau$. Then $\|\Proj_{V'}(\vec{u})\|\geq|\langle\vec{u},\vec{u}'\rangle|>1-\tau$. Suppose $\vec{u}=\lambda_1\vec{u}_1+\lambda_2\vec{u}_2+\cdots+\lambda_{k_1}\vec{u}_{k_1}$, where $\lambda_1^2+\lambda_2^2+\cdots+\lambda_{k_1}^2=1$. We have
\begin{equation*}\begin{split}
(1-\tau)^2 & <\|\Proj_{V'}(\vec{u})\|^2\leq\Big(\sum_{j=1}^{k_1}|\lambda_j|\cdot\|\Proj_{V'}(\vec{u}_j)\|\Big)^2\leq\Big(\sum_{j=1}^{k_1}\lambda_j^2\Big)\Big(\sum_{j=1}^{k_1}\|\Proj_{V'}(\vec{u}_j)\|^2\Big) \\
& = \sum_{j=1}^{k_1}\|\Proj_{V'}(\vec{u}_j)\|^2.
\end{split}\end{equation*}
Therefore there exists $j\in[k_1]$ such that $\|\Proj_{V'}(\vec{u}_j)\|^2\geq(1-\tau)^2/k_1$.
\end{proof}

We will need the following lower bound for the rank of a diagonal dominating matrix. The same lemma for Hermitian matrices was proved in~\cite{BDWY-pnas}. Here we change the proof slightly and show that the consequence also holds for an arbitrary matrix.

\begin{lemma} \label{lem:diagdom}
Let $D=(d_{ij})$ be a complex $m\times m$ matrix and $L,K$ be positive real numbers. If $d_{ii}=L$ for every $i\in[m]$ and $\sum_{i\neq j}|d_{ij}|^2\leq K$, then $\rank(D)\geq m-K/L^2$.
\end{lemma}

\begin{proof}
Let $r$ be the rank of $D$. Consider the singular value decomposition of $D$, say $D=U\Sigma V$, where $U,V$ are unitary matrices and $\Sigma$ is a non-negative diagonal matrix. Let $\sigma_1,\sigma_2,\ldots,\sigma_r$ be the nonzero singular values on the diagonal of $\Sigma$.
\begin{equation*}\begin{split}
(mL)^2&=\tr(D)^2=\tr(U\Sigma V)^2=\tr(\Sigma(VU))^2\leq(\sigma_1+\cdots+\sigma_r)^2\leq r(\sigma_1^2+\cdots+\sigma_r^2) \\
&=r\|D\|_F^2\leq r(mL^2+K).
\end{split}\end{equation*}
Therefore $r\geq(mL)^2/(mL^2+K)=m^2/(m+K/L^2)\geq m-K/L^2$.
\end{proof}

The following theorem handles the `well separated case' of Theorem~\ref{thm:main_}.

\begin{theorem} \label{thm:sep}
Let $\mathcal{V}=(V_1,V_2,\ldots,V_n)$ ($V_i\in\R^\ell$) be a list of $k$-bounded vector spaces with an $(\alpha,\delta)$-system $\mathcal{S}=(S_1,S_2,\ldots,S_w)$ and $d=\dim(V_1+V_2+\cdots+V_n)$. If for every $j\in[w]$ and $\{i_1,i_2\}\subseteq S_j$, $V_{i_1}$ and $V_{i_2}$ are $\tau$-separated, then $d\leq\alpha k/(\tau\delta)$.
\end{theorem}

\begin{proof}
Let $k_1,k_2,\ldots,k_n$ be the dimensions of $V_1,V_2,\ldots,V_n$, and $m=k_1+k_2+\cdots+k_n$. For every $i\in[n]$, fix $B_i=\{\vec{u}_{i1},\vec{u}_{i2},\ldots,\vec{u}_{i{k_i}}\}$ to be some orthonormal basis of $V_i$. We use $A$ to denote the $m\times\ell$ matrix whose rows are $\vec{u}_{11}^T,\ldots,\ldots,\vec{u}_{nk_n}^T$. We will bound $d=\rank(A)$ by constructing a high rank $m\times m$ matrix $D$ satisfying $DA=0$.

For $s\in[m]$, we use $\psi(s)\in[n]$ to denote the number satisfying
$$k_1+k_2+\cdots+k_{\psi(s)-1}+1\leq s\leq k_1+k_2+\cdots+k_{\psi(s)-1}+k_{\psi(s)}.$$
In other words, the $s$-th row of $A$ is a vector in $B_{\psi(s)}$.

\begin{claim}
For every $s\in[m]$, there is a vector $\vec{y}_s\in\R^m$ satisfying $\vec{y}_s^TA=\vec{0}^T$, $y_{ss}=\lceil\delta n\rceil$, and $\sum_{t\neq s}y_{st}^2\leq\alpha\lceil\delta n\rceil/\tau$.
\end{claim}

\begin{proof}
Say the $s$-th row of $A$ is $\vec{u}^T$, where $\vec{u}\in B_{\psi(s)}$. Let $J\subseteq[w]$ be a set of size $|J|=\lceil\delta n\rceil$ such that for every $j\in J$, $S_j$ contains $\psi(s)$. We construct a vector $\vec{c}_j$ for every $j\in J$ as following.
\begin{itemize}
\item If $S_j$ contains 3 elements $\{\psi(s),i,i'\}$, we have $\lambda_1,\lambda_2,\ldots,\lambda_{k_i},\mu_1,\mu_2,\ldots,\mu_{k_{i'}}\in\R$ such that
$$\vec{u}-\lambda_1\vec{u}_{i1}-\lambda_2\vec{u}_{i2}-\cdots-\lambda_{k_i}\vec{u}_{ik_i}-\mu_1\vec{u}_{i'1}-\mu_2\vec{u}_{i'2}-\cdots-\mu_{k_{i'}}\vec{u}_{i'k_{i'}}=\vec{0}.$$
We can obtain from this equation a vector $\vec{c}_j$ such that $\vec{c}_j^TA=\vec{0}^T$, $c_{js}=1$, and by Lemma~\ref{lem:smallcoef}
$$\sum_{t\neq s}c_{jt}^2=\lambda_1^2+\lambda_2^2+\cdots+\lambda_{k_i}^2+\mu_1^2+\mu_2^2+\cdots+\mu_{k_{i'}}^2\leq\frac{1}{\tau}.$$
\item If $S_j$ contains 2 elements $\{\psi(s),i\}$, there exist $\lambda_1,\lambda_2,\ldots,\lambda_{k_i}$ with $\lambda_1^2+\lambda_2^2+\cdots+\lambda_{k_i}^2=1$ such that
$$\vec{u}-\lambda_1\vec{u}_{i1}-\lambda_2\vec{u}_{i2}-\cdots-\lambda_{k_i}\vec{u}_{ik_i}=\vec{0}.$$
We can obtain from this equation a vector $\vec{c}_j$ such that $\vec{c}_j^TA=\vec{0}^T$, $c_{js}=1$, and
$$\sum_{t\neq s}c_{jt}^2=\lambda_1^2+\lambda_2^2+\cdots+\lambda_{k_i}^2=1\leq1/\tau.$$
\end{itemize}
In either case we obtain a $\vec{c}_j$ such that $\vec{c}_j^TA=\vec{0}^T$, $c_{js}=1$ and $\sum_{t\neq s}c_{jt}^2\leq1/\tau$. We define 
$$\vec{y}_s=\sum_{j\in J}\vec{c}_j.$$
We have $\vec{y}_s^TA=\vec{0}^T$ and $y_{ss}=\lceil\delta n\rceil$. We consider $\sum_{t\neq s}y_{st}^2$. From the above construction of $\vec{c}_j$, we can see $c_{jt}\neq0$ ($t\neq s$) only when $\psi(t)\neq\psi(s)$ and $\{\psi(s),\psi(t)\}\subseteq S_j$. Hence for every $t\neq s$, there are at most $\alpha$ nonzero values in $\{c_{jt}\}_{j\in J}$. It follows that
$$\sum_{t\neq s}y_{st}^2=\sum_{t\neq s}\left(\sum_{j\in J}c_{jt}\right)^2\leq\alpha\sum_{t\neq s}\left(\sum_{j\in J}c_{jt}^2\right)=\alpha\sum_{j\in J}\left(\sum_{t\neq s}c_{jt}^2\right)\leq\frac{\alpha\lceil\delta n\rceil}{\tau}.$$
Thus the claim is proved.
\end{proof}

Define $D$ to be the matrix consists of rows $\vec{y}_1^T,\vec{y}_2^T,\ldots,\vec{y}_m^T$. Then every entry on the diagonal of $D$ is $\lceil\delta n\rceil$, and the sum of squares of all entries off the diagonal is at most $\alpha\lceil\delta n\rceil m/\tau$. Apply Lemma~\ref{lem:diagdom} on $D$, and we have
$$\rank(D)\geq m-\frac{\alpha\lceil\delta n\rceil m/\tau}{\lceil\delta n\rceil^2}=m-\frac{\alpha m}{\tau\lceil\delta n\rceil}\geq m-\frac{\alpha k}{\tau\delta}.$$
By $DA=0$, the rank of $A$ is $d\leq\alpha k/(\tau\delta)$.
\end{proof}

\subsection{Proof of Theorem~\ref{thm:main_} -- general case}

Now we prove Theorem~\ref{thm:main_}. We assume that the first case of Theorem~\ref{thm:main_} does not hold, i.e. $d>400\alpha k^3/(\beta\delta)$. We will show the second case holds.

\begin{lemma} \label{lem:bigprob}
At least one of the following two cases holds:
\begin{enumerate}
\item The second case of Theorem~\ref{thm:main_} holds, i.e. there exists a sublist of $q\geq\delta n/(20\alpha)$ spaces $(V_{i_1},V_{i_2},\ldots,V_{i_q})$ such that there are nonzero vectors $\vec{z}_1\in V_{i_1},\vec{z}_2\in V_{i_2},\ldots,\vec{z}_q\in V_{i_q}$ with 
$$\dim(\vec{z}_1,\vec{z}_2,\ldots,\vec{z}_q)\leq\beta d.$$
\item There exist a distribution $\mathcal{D}$ on $\mathcal{V}$-admissible sets and an $I\subseteq[n]$ with $|I|\geq(1-\delta/(10\alpha))n$ such that for every $i\in I$,
$$\Pr_{H\sim\mathcal{D}}[i\in H]\geq\frac{\beta d}{4kn}.$$
\end{enumerate}
\end{lemma}

\begin{proof}
 We sample a $\mathcal{V}$-admissible set as follows: Initially let $F=\emptyset$. In each step we pick a space $V_{i_0}$ among $V_1,V_2,\ldots,V_n$ with $V_{i_0}\bigcap\sum_{i\in F}V_i=\{\vec{0}\}$, and add $i_0$ to $F$. If such a $V_{i_0}$ does not exist, the procedure terminates. Let $H$ be the final value of $F$. Clearly, $H$ is $\mathcal{V}$-admissible. Let $\mathcal{D}$ be the distribution of $H$. We will show that if there does not exist an $I\subseteq[n]$ such that $\mathcal{D}$ and $I$ satisfy the second case, the first case must hold.

In the above random procedure, if it is possible that $\dim(\sum_{i\in H}V_i)\leq\beta d$, then there are nonzero vectors $\vec{z}_1\in V_1$, $\vec{z}_2\in V_2$, \ldots, $\vec{z}_n\in V_n$ contained in $\sum_{i\in H}V_i$, which has dimension at most $\beta d$. Since $\delta/\alpha\leq3/2$ by Lemma~\ref{lem:awd}, we have $n\geq\delta n/(20\alpha)$ and the lemma is proved. In the remaining proof we assume that $H$ always satisfies
$\dim\left(\sum_{i\in H}V_i\right)>\beta d.$
This implies that there are always at least $\beta d/k$ elements in $H$. Fix $t=\lceil \beta d/(2k)\rceil<|H|$ (recall our assumption $d>400\alpha k^3/(\beta\delta)$), and let $V_{j_1},V_{j_2},\ldots,V_{j_t}$ be the first $t$ spaces.

We assume that the second case of the lemma does not hold, i.e. there is a set $X$ of at least $\delta n/(10\alpha)$ $i$'s with $\Pr[i\in H]\leq \beta d/(4kn)\leq t/(2n)$. We will show the first case holds.

\begin{claim}
For every $i\in X$, $\Pr\big[V_i\cap(V_{j_1}+V_{j_2}+\cdots+V_{j_t})\neq\{\vec{0}\}\big]\geq\frac{1}{2}$.
\end{claim}

\begin{proof}
The proof is similar to Claim~6.4 in~\cite{DSW14}. There are 3 disjoint events
\begin{eqnarray*}
E_1 & : &  V_i\cap(V_{j_1}+V_{j_2}+\cdots+V_{j_t})=\{\vec{0}\}, \\
E_2 & : &  i\in\{j_1,j_2,\ldots,j_t\},\text{ i.e. $i$ is picked in the first $t$ steps}, \\
E_3 & : &  i\notin\{j_1,j_2,\ldots,j_t\}\text{ and }V_i\cap(V_{j_1}+V_{j_2}+\cdots+V_{j_t})\neq\{\vec{0}\}.
\end{eqnarray*}
As long as $i$ is not picked, the $s$-th ($s\in[t]$) element $j_s$, conditioned on $ E_1\cup E_2$, is sampled uniformly at random from $([n]\setminus J_s)\cup\{i\}$, where $J_s=\{j\in[n]:V_j\cap(V_i+V_{j_1}+V_{j_2}+\cdots+V_{j_{s-1}})\neq\{\vec{0}\}$. Therefore the probability that $V_i$ is not picked in the first $t$ steps conditioning on $E_1\cup E_2$ is
$$\Pr[E_1\mid E_1\cup E_2]\leq(1-\frac{1}{n})(1-\frac{1}{n-1})\cdots(1-\frac{1}{n-t+1})=\frac{n-t}{n}.$$
Hence $t/(2n)\geq\Pr[i\in H]\geq\Pr[E_2]\geq(t/n)\Pr[E_1\cup E_2]$. It follows that
$$\Pr\Big[V_i\cap(V_{j_1}+V_{j_2}+\cdots+V_{j_t})\neq\{\vec{0}\}\Big]\geq\Pr[E_3]=1-\Pr[E_1+E_2]\geq1-\frac{1}{2}=\frac{1}{2}.$$
Thus the claim is proved.
\end{proof}

Therefore the expected number of $i$'s in $X$ with $V_i\cap(V_{j_1}+V_{j_2}+\cdots+V_{j_t})\neq\{\vec{0}\}$ is at least $|X|/2\geq\delta n/(20\alpha)$. Each of these $V_i$'s has a nonzero vector contained in $V_{j_1}+V_{j_2}+\cdots+V_{j_t}$. The first case is proved by $\dim(V_{j_1}+V_{j_2}+\cdots+V_{j_t})\leq kt\leq\beta d$.
\end{proof}

To prove Theorem~\ref{thm:main_}, we only need to consider the second case in Lemma~\ref{lem:bigprob}. Let $p_i$ ($i\in[n]$) be the probability that $i$ is contained in $H\sim\mathcal{D}$, and $I\subseteq[n]$ be the set such that $|I|\geq(1-\delta/(10\alpha))n$ and $p_i\geq \beta d/(4kn)$ for every $i\in I$. We use $k_1,k_2,\ldots,k_n$ to denote the dimensions of $V_1,V_2,\ldots,V_n$.

\begin{lemma}
The vector $\vec{p}=(p_1,p_2,\ldots,p_n)$ is in the convex hull of $\mathcal{V}$-admissible vectors.
\end{lemma}

\begin{proof}
For every $\mathcal{V}$-admissible set $H$, we use $q_H$ to denote the probability that $H$ is picked according to $\mathcal{D}$, and $\vec{1}_H$ to denote the $\mathcal{V}$-admissible vector corresponding to $H$. Then,
$$\vec{p}=(p_1,p_2,\ldots,p_n)=\sum_{\mathcal{V}\text{-admissible }H}q_H\vec{1}_H$$
and
$p_i$ is exactly the probability that $i \in H$.
\end{proof}

We apply Theorem~\ref{thm:barthec} with the $\vec{p}=(p_1,p_2,\ldots,p_n)$, and obtain an invertible linear map $M:\R^\ell\mapsto\R^\ell$ such that for any unit vector $\vec{w}\in\R^\ell$,
$$\sum_{i=1}^np_i\|\Proj_{V_i'}(\vec{w})\|^2\leq2,$$
where $V_i'$ denotes $M(V_i)$. Since $p_i\geq \beta d/(4kn)$ for every $i\in I$, we have
\begin{equation} \label{eqn:bartheapp}
\sum_{i\in I}\|\Proj_{V_i'}(\vec{w})\|^2\leq\frac{8kn}{\beta d}.
\end{equation}

We will reduce the problem to the special case discussed in the previous subsection. We say a pair $\{i_1,i_2\}\subseteq[n]$ is {\em bad} if $V_{i_1}',V_{i_2}'$ are not $0.5$-separated. Let $\mathcal{S}=(S_1,S_2,\ldots,S_w)$ be the $(\alpha,\delta)$-system of $\mathcal{V}$. By Lemma~\ref{lem:linearmap}, $\mathcal{S}$ is also an $(\alpha,\delta)$-system of $\mathcal{V}'=(V_1',V_2',\ldots,V_n')$. We estimate the number of sets among $S_1,S_2,\ldots,S_w$ containing a bad pair.

\begin{lemma}
For every $i_0\in I$, there at most $\delta n/(10\alpha)$ values of $i\in I$ such that $V_{i_0}'$ and $V_i'$ are not $0.5$-separated.
\end{lemma}

\begin{proof}
Let $\{\vec{u}_1,\vec{u}_2,\ldots,\vec{u}_{k_{i_0}}\}$ be an orthonormal basis of $V_{i_0}'$. For any $i$ that $V_{i_0}'$ and $V_i'$ are not $0.5$-separated, by Lemma~\ref{lem:spacetobasis}, there must be $j\in[k_{i_0}]$ such that
$$\|\Proj_{V_i'}(\vec{u}_j)\|^2\geq\frac{1}{4k_{i_0}}\geq\frac{1}{4k}.$$
For every $j_0\in[k_{i_0}]$, we set $\vec{w}=\vec{u}_{j_0}$ in inequality~(\ref{eqn:bartheapp}). The number of $i$'s that $\|\Proj_{V_i'}(\vec{u}_{j_0})\|\geq1/(4k)$ is at most
$$\frac{8kn}{\beta d}\left/\frac{1}{4k}\right.=\frac{32k^2n}{\beta d}.$$
Since there are $k_{i_0}\leq k$ values of $j_0\in[k_{i_0}]$, the number of $i$'s that $V_{i_0}'$ and $V_i'$ are not $0.5$-separated is at most
$$k\cdot\frac{32k^2n}{\beta d}\leq\frac{32k^3n}{\beta d}\leq\frac{\delta n}{10\alpha}.$$
In the last inequality we used the assumption $d>400\alpha k^3/(\beta \delta)$.
\end{proof}

The number of bad pairs is at most
$$|[n]\setminus I|\cdot n+|I|\cdot\frac{\delta n}{10\alpha}\leq\frac{\delta n^2}{10\alpha}+\frac{\delta n^2}{10\alpha}=\frac{\delta n^2}{5\alpha}.$$
We remove all $S_j$'s that contains a bad pair and use $\mathcal{S}'$ to denote the list of the remaining sets. Since each pair appears at most $\alpha$ times, we have removed at most $\delta n^2/5$ sets. Originally we have at least $\delta n^2/3$ sets by Lemma~\ref{lem:awd}. Now we have at least $\delta n^2/3-\delta n^2/5\geq\delta n^2/10$ sets. By Lemma~\ref{lem:removebad}, there is a sublist $\mathcal{V}''=(V_{i_1}',V_{i_2}',\ldots,V_{i_q}')$ ($q\geq\delta n/(20\alpha)$) of $\mathcal{V}'$ and a sublist $\mathcal{S}''$ of $\mathcal{S}'$ such that $\mathcal{S}''$ is an $(\alpha,\delta/20)$-system of $\mathcal{V}''$.

Since we have removed all bad pairs, $\mathcal{V}''$ and $\mathcal{S}''$ must satisfy the conditions of Theorem~\ref{thm:sep}. By Theorem~\ref{thm:sep},
$$\dim(V_{i_1}'+V_{i_2}'+\cdots+V_{i_q}')\leq\frac{\alpha k}{0.5\cdot\delta/20}=\frac{40\alpha k}{\delta}\leq\beta d.$$
In the last inequality we used the assumption $d>400\alpha k^3/(\beta \delta)$. Recall that the linear map $M$ is invertible. So the space $V_{i_1}+V_{i_2}+\cdots+V_{i_q}$ has the same dimension as $V_{i_1}'+V_{i_2}'+\cdots+V_{i_q}'$. Therefore there are $q\geq\delta n/(20\alpha)$ spaces $V_{i_1},V_{i_2},\ldots,V_{i_q}$ within dimension $\beta d$. The second case of Theorem~\ref{thm:main_} holds. In summary, under the assumption $d>400\alpha k^3/(\beta\delta)$ we have shown the second case of Theorem~\ref{thm:main_} is always satisfied. Therefore Theorem~\ref{thm:main_} is proved. \hfill$\Box$

\bibliographystyle{alpha}
\bibliography{hidimSG}

\end{document}